\documentclass[11pt]{amsart}

\usepackage{amssymb
,amsthm
,amsmath
,amscd
,mathtools}
\usepackage{mathrsfs}
\usepackage{enumitem}
\usepackage[all]{xypic}
\usepackage[bookmarks=true,bookmarksnumbered=true]{hyperref}
\usepackage[hmargin=2cm,vmargin=3cm]{geometry}
\usepackage[all]{xy}
\allowdisplaybreaks

\numberwithin{equation}{section}

\theoremstyle{plain}
\newtheorem{thm}{Theorem}[section]

\newtheorem*{GPconj}{Gan--Gross--Prasad conjecture}
\newtheorem*{main}{Main Theorem}

\theoremstyle{definition}

\theoremstyle{remark}
\newtheorem{rem}[thm]{Remark}
\newtheorem{con}{Conjecture}

\newcommand{\half}{\frac{1}{2}}

\def\disc{\operatorname{disc}}

\def\Hom{\operatorname{Hom}}

\def\Irr{\operatorname{Irr}}

\def\N{\operatorname{N}}

\def\Tr{\operatorname{Tr}}
\def\beqn{\begin{equation}}
\def\eeqn{\end{equation}}
\def\beqnn{\begin{equation*}}
\def\eeqnn{\end{equation*}}
\def\beqna{\begin{eqnarray}}
\def\eeqna{\end{eqnarray}}
\def\beqnan{\begin{eqnarray*}}
\def\eeqnan{\end{eqnarray*}}

\def\GL{\mathrm{GL}}

\def\SL{\mathrm{SL}}

\def\U{\mathrm{U}}

\def\WD{\mathit{WD}}

\def\CC{\mathbb{C}}

\def\ZZ{\mathbb{Z}}

\def\1{\Eins}

\makeatletter
\def\varddots{\mathinner{\mkern1mu
    \raise\p@\hbox{.}\mkern2mu\raise4\p@\hbox{.}\mkern2mu
    \raise7\p@\vbox{\kern7\p@\hbox{.}}\mkern1mu}}
\makeatother

\newcommand{\ep}{\varepsilon}
\newcommand{\BIGOP}[1]{\mathop{\mathchoice%
{\raise-0.22em\hbox{\huge $#1$}}%
{\raise-0.05em\hbox{\Large $#1$}}{\hbox{\large $#1$}}{#1}}}

\newcommand{\BIGboxplus}{\mathop{\mathchoice%
{\raise-0.35em\hbox{\huge $\boxplus$}}%
{\raise-0.15em\hbox{\Large $\boxplus$}}{\hbox{\large $\boxplus$}}{\boxplus}}}

\title{The local Gan-Gross-Prasad conjecture for $U(n+1) \times U(n)$ : a non-generic case}
\author{Jaeho Haan}
\address{Center for Mathematical Challenges,
Korea Institute for Advanced Study,
85 Hoegiro Dongdaemun-Gu,
Seoul 130-722,
South Korea}
\email{jaehohaan@gmail.com}
\date{\today}

\begin{document}

\begin{abstract}The local Gan-Gross-Prasad conjecture of unitary groups, which is now settled by the works of Plessis, Gan and Ichino, says that for a pair of generic $L$-parameters of $(U(n+1),U(n))$, there is a unique pair of representations in their associated Vogan $L$-packets which produces the Bessel model. In this paper, we examined the conjecture for a pair of $L$-parameters of $\big(U(n+1), U(n)\big)$ as fixing a special non-generic parameter of $U(n+1)$ and varing tempered $L$-parameters of $U(n)$ and observed that there still exist a Gan-Gross-Prasad type formulae depending on the choice of $L$-parameter of $U(n)$.

\end{abstract}
\maketitle
\section{\textbf{Introduction}}The local  \emph{Gan-Gross-Prasad} (GGP)  conjecture concerns the restriction problem of real or $p$-adic Lie groups. Though the GGP conjecture is now formulated for all classical groups, we will restrict ourselves only to unitary groups in this paper.

Let $E/F$ be a quadratic extension of local fields of characteristic zero. Let $V_{n+1}$ be a Hermitian space of dimension $n+1$ over $E$ and $W_n$ a skew-Hermitian space of dimension $n$ over $E$.
Let $V_n \subset V_{n+1}$  be a nondegenerate subspace of codimension $1$ and we set 
\[  G_n =  \U(V_n) \times \U(V_{n+1}) \quad \text{or} \quad \U(W_n) \times \U(W_n) \]
and
\[   H_n = \U(V_n) \quad \text{or} \quad \U(W_n).\]
Then we have a diagonal embedding
\[ \Delta:  H_n \hookrightarrow G_n. \]

Let $\pi$ be an irreducible smooth representation of $G_n$. In the Hermitian case, one is interested in computing 
\[  \dim_\CC \Hom_{\Delta H_n} ( \pi, \CC)\]
and it is called the \emph{Bessel} case (B) of the GGP conjecture.
To describe the GGP conjecture for the skew-Hermitian case, we need another data, that is a Weil representation $\omega_{\psi, \chi, W_n}$. (Here, $\psi$ is a nontrivial additive character of $F$ and $\chi$ is a character of $E^{\times}$ whose restriction to $F^{\times}$ is the non-trivial quadratic character associated to $E/F$ by local class field theory.)
In this case,  one is interested in computing
\[  \dim_\CC \Hom_{\Delta H_n} ( \pi, \omega_{\psi,\chi, W_n}) \]
and we call this the \emph{Fourier--Jacobi} case (FJ) of the GGP conjecture. To treat them simultaneously, we use the notation $\nu = \CC$ or $\omega_{\psi,\chi, W_n}$ in the respective cases.

By the results of \cite{agrs} and \cite{sun}, it is known \[  \dim_\CC \Hom_{\Delta H_n} ( \pi, \nu ) \le 1. \] So our next task should be specifying irreducible smooth representations $\pi$ such that $$\Hom_{\Delta H_n} ( \pi, \nu ) = 1. $$

 In a seminal paper \cite{Gan2}, Gan, Gross and Prasad proposed a conjecture which contains both (partial) mulitiplicity one theorem and the answer to the above question. 
To explain it, we need the notion of relevant pure inner forms of $G_n$ and relevent Vogan $L$-packets.  A pure inner form of $G_n$ is a group of the form 
\[  G_n' = \U(V_{n+1}') \times \U(V'_{n}) \quad \text{or} \quad \U(W'_n) \times \U(W'_n) \]
where $V_{n}' \subset V_{n+1}'$ are hermitian spaces over $E$ whose dimensions are $n$ and $n+1$ respectively and $W_n'$ is a $n$-dimensional skew-hermitian spaces over $E$.

\noindent Furthermore, if \[  \quad  V_{n+1}'/V_{n}' \cong V_{n+1}/V_{n} \quad \text{or} \quad W_n'=W_n'', \] we say that $G_n'$ is a relevant pure inner form of $G_n$.

If $G_n'$ is relevant of $G_n$, we set
\[   H'_n = \U(V'_n) \quad \text{or} \quad \U(W'_n) \]
so that we have a diagonal embedding
\[ \Delta:  H'_n \hookrightarrow G'_n. \]

For an $L$-parameter $\phi$ of $G_n$, there is the associated (relevant) Vogan $L$-packet $\Pi_{\phi}$  which consists of certain irreducible smooth representations of $G_n$ and its (relevant) pure inner forms $G_n'$ whose corresponding $L$-parameter is $\phi$. We denote the relevant Vogan $L$-packet of $\phi$ by $\Pi^R_{\phi}$.

Now we can loosely state the GGP conjecture  as follows:

\begin{GPconj}
For a generic $L$-parameter $\phi$ of $G_n$, the followings hold:
\begin{enumerate}
\item $\sum_{\pi' \in \Pi^R_{\phi}}\dim_{\CC}\Hom_{\Delta H'_n} ( \pi', \nu )=1.$

\item Using the local Langlands correspondence for unitary group, we can pinpoint $\pi' \in \Pi^R_{\phi}$ such that $$\dim_{\CC}\Hom_{\Delta H'_n} ( \pi', \nu )=1.$$ 
\end{enumerate}
\end{GPconj}

Following the strategy of Waldspurger (\cite{W2}--\cite{W5}) for orthogonal groups, Beuzart-Plessis \cite{bp1},\cite{bp2},\cite{bp3} established (B) of the GGP conjecture for tempered $L$-parameter $\phi$.  Building upon Plessis's work, Gan and Ichino \cite{iw} proved (FJ) for tempered case first by establishing the precise local theta correspondence for almost equal rank unitray groups and then extended both (B) and (FJ) to generic cases. Because the generic case is now completely settled, it is natural to turn our attention to the non-generic case.

In \cite{Haan}, the author considered a non-generic case of (B) when $n=2$. This paper can be seen as an extension of the result, because we shall investigate non-generic case of (B) for all $n \ge 2$ when an $L$-parameter of $G_{n}$ involves some non-generic $L$-parameter of $U(V_{n+1})$. We roughly state our main result in the following.

\begin{main}For all $n\ge 1$, let $\phi^{NG}$ be a special non-generic $L$-parameter of $U(V_{n+2})$ obtained from the theta lift of a certain $L$-parameter of $U(V_n)$ and $\phi^T$ be a tempered $L$-parameter of $U(V_{n+1})$. Then for the $L$-parameter $\phi=\phi^{NG} \otimes \phi^T$ of $G_{n+1}=U(V_{n+2}) \times U(V_{n+1})$, we have 

\begin{enumerate}
\item If the $L$-parameter $\phi^T$ does not contain $\chi_{W}^{-1}$, $$\sum_{\pi' \in \Pi^R_{\phi}}\dim_{\CC}\Hom_{\Delta H'_{n+1}} ( \pi', \CC )=0$$
\item Suppose that $\phi^T$ contains $\chi_{W}^{-1}$. Then $$\sum_{\pi' \in \Pi^R_{\phi}}\dim_{\CC}\Hom_{\Delta H'_{n+1}} ( \pi', \CC )\ge1.$$ 

\item If the multiplicity of $\chi_{W}^{-1}$ in $\phi^T$ is one, we have  $$\sum_{\pi' \in \Pi^R_{\phi}}\dim_{\CC}\Hom_{\Delta H'_{n+1}} ( \pi', \CC )=1.$$ Furthermore, using the local Langlands correspondence,  we can explicitly describe $\pi' \in \Pi^R_{\phi}$ such that \begin{equation}\label{p}\dim_{\CC}\Hom_{\Delta H'_{n+1}} ( \pi', \CC )=1.\end{equation}
\end{enumerate}

\end{main}

The rest of the paper is organized as follows; In Section 2 and 3, we give a brief summary on the local Langlands correspondence and the local theta correspondence for unitary groups. After describing these background materials,  we prove  our main Theorem \ref{thm2} in Section 4.

\subsection{Notation}\label{not}We fix some notations we shall use throughout this paper:
\begin{itemize}
\item  $E/F$ is a quadratic extension of local fields of characteristic zero.

\item $\text{Fr}_E$ is a Frobenius element of Gal$(\bar{E}/E)$.

\item  $c$ is the non-trivial element of Gal$(E/F)$.

\item  The trace and norm maps from $E$ to $F$ are denoted by $\Tr_{E/F}$ and $\N_{E/F}$ respectively.

\item $\delta$ is an element of $E^{\times}$ such that  $\Tr_{E/F}(\delta)=0$.

\item $\psi$ is an additive character of $F$.

\item  $\omega_{E/F}$ is the non-trivial quadratic character assosiated to $E/F$ by local class field theory.

\item For an linear algebraic group $G$, denote its $F$-points by $G(F)$.

\end{itemize}

\section{\textbf{Local Langlands correspondence for unitary group}}
The local Langlands correpondence (LLC) for unitary groups, which parametrizes irreducible smooth representations of $U(n)$, is now known by the work of Mok \cite{Mok} and Kaletha-M\'inguez-Shin-White \cite{kmsw}  under some assumption on the weighted fundamental lemma. Since the GGP conjecture and our main results are both expressed using the LLC, we shall assume the LLC for unitary group. In this section, we list some of its properties we shall use in this paper. Note that much of this section are excerpts from Section.2 in \cite{iw}.

\subsection{Hermitian and skew-Hermitian spaces}
Fix $\varepsilon \in \{\pm1\}$. Let $V$ be a finite $n$-dimensional vector space over $E$ equipped with a nondegenerate $\varepsilon$-hermitian $c$-sesquilinear form $\langle \cdot, \cdot \rangle_V : V \times V \rightarrow E$.
It means that for $v, w \in V$ and $a, b \in E$, the following holds : \[
 \langle a v, b w \rangle_V = a b^c \langle v, w \rangle_V, \qquad
 \langle w, v \rangle_V = \varepsilon \cdot \langle v, w \rangle_V^c.
\]

\noindent We define $\disc V $ by $ (-1)^{(n-1)n/2} \cdot \det V$ so that 
\[
 \disc V \in
 \begin{cases}
  F^{\times} / \mathrm{N}_{E/F}(E^{\times})
  & \text{if $\varepsilon = +1$;} \\
  \delta^n \cdot F^{\times} / \mathrm{N}_{E/F}(E^{\times})
  & \text{if $\varepsilon = -1.$}
 \end{cases}
\]
Using $\disc V$, we  define $\epsilon(V) \in \{\pm 1\}$ by 
\begin{equation}\label{sign}
 \epsilon(V) = 
 \begin{cases}
  \omega_{E/F}(\disc V) & \text{if $\varepsilon = +1$;} \\
  \omega_{E/F}(\delta^{-n} \cdot \disc V) & \text{if $\varepsilon = -1$.}
 \end{cases}
\end{equation}
 For a given positive integer $n$, it is known (by a theorem of Landherr) that there are exactly two isomorphism classes of $\varepsilon$-hermitian spaces of dimension $n$ and they are distinguished from each other by $\epsilon(V)$\\
The unitary group of $V$ is defined by
\[
  \U(V) = \{ g \in \GL(V) \, | \,
 \text{$\langle g v, g w \rangle_V =  \langle v, w \rangle_V$ for $v, w \in V$}
 \}
\]
and it  turns out to be connected reductive algebraic group defined over $F$.
\subsection{$L$-parameters and component groups}

Let $I_F$ and $\text{Fr}_F$ be the inertia subgroup and Frobenious element of $\text{Gal}(\bar{F}/F)$ respectively.  Let $W_F=I_F \rtimes \langle \text{Fr}_F \rangle $ be the Weil group of $F$  and $\WD_F = W_F \times \SL_2(\CC)$ the Weil-Deligne group of $F$.
We say that a homomorphism $\phi: \WD_F \rightarrow \GL_n(\CC)$ is a representation of $\WD_F$ if 
\begin{enumerate}
\item $\phi(\text{Fr}_F)$ is semisimple
\item $\phi$ is continuous
 \item the restriction of $\phi$ to $\SL_2(\CC)$ is induced by a morphism of algebraic groups $\SL_2 \to \GL_n$
\end{enumerate}
We call $\phi$ is tempered if the image of $W_F$ is bounded. The contragredient representation $\phi^{\vee}: \WD_F \rightarrow \GL_n(\CC)$ of $\phi$ is defined by $$\phi^\vee(w) := {}^t\phi(w)^{-1} \text{  for all $w\in WD_F$. }$$
We choose $s \in W_F \smallsetminus W_E$. The Asai representation $\text{As}(\phi):WD_F \to \GL_{n^2}(\mathbb{C})$ of $\phi$ is defined as follows;
$$\text{As}(\phi)(w)=\begin{cases}\phi(w)\otimes \phi(s^{-1}ws) \quad \quad  \quad \quad\text{ if } w\in WD_E \\ \iota \circ (\phi(s^{-1}w) \otimes \phi(ws)) \quad \quad \text{ if } w \in WD_F \smallsetminus WD_E\end{cases}$$ where $\iota$ is the linear isomorphism of $\mathbb{C}^n\otimes \mathbb{C}^n$ given by $\iota(x \otimes y )=y \otimes x$. Note that the equivalence class of $\text{As}(\phi)$ is independent of the choice of $s$. We denote $\text{As}^+(\phi)$ by $\text{As}(\phi)$ and $\text{As}^{-}(\phi)$ by $\text{As}(\chi  \otimes \phi)$. 

 If $\phi$ is a representation of $\WD_E$, we define a representation $\phi^c: \WD_E \rightarrow \GL_n(\CC)$ by $\phi^c(w) = \phi(sws^{-1})$  for all $w\in WD_E$.
(Note that the equivalence class of $\phi^c$ is independent of the choice of $s$.) We say that $\phi$ is conjugate self-dual if there is an isomorphism $b: \phi \mapsto (\phi^{\vee})^{c}$. Since there is a natural isomorphism $(((\phi^{\vee})^c)^{\vee})^c\simeq \phi$, we can consider $(b^{\vee})^c$ as an isomorphism from $\phi$ onto $(\phi^{\vee})^c$. For $\varepsilon \in \{\pm 1\}$,  if there exists such an isomorphism $b$ satisfying the extra condition $(b^{\vee})^c=\varepsilon \cdot b$, we call $\phi$ conjugate self-dual with sign $\varepsilon$.

Let $V$ be a $n$-dimensional $\varepsilon$-hermitian space over $E$. An $L$-parameter for the unitary group $\U(V)$ is an equivalence classes of conjugate self-dual representations $\phi:\WD_E \longrightarrow GL_n(\CC)$ of sign $(-1)^{n-1}$.
Given a $L$-parameter $\phi$ of $U(V)$, we say that  $\phi$ is generic if its Asai $L$-function $L(s,\text{As}^{(-1)^{n-1}} (\phi))$ is holomorphic at $s=1$.
We decompose $\phi$ as a direct sum
\[  \phi = m_i \phi_i + \cdots + m_r \phi_r + \phi'+  {}^c\phi'^{\vee}\]
where $\phi_i$ are distinct irreducible conjugate self-dual representations of $\WD_E$ with the same type as $\phi$ and multiplicitys $m_i$ and $\phi'$ is a sum of irreducible conjugate self-dual representations not of the same type as $\phi$. 
We say that $\phi$ is discrete if $m_i=1$ for all $i$ and $\phi'$ does not appear. (i.e. all irreducible summand of $\phi$ are of same type as $\phi$.) For an $L$-parameter $\phi$, we can associate its component group $S_{\phi}$ as follows:
\[  S_{\phi}  = \prod_j  (\ZZ / 2\ZZ) a_j. \]
Namely, $S_{\phi}$ is a free  $\mathbb{Z}/2\mathbb{Z}$-module of rank $r$ with a  canonical basis $\{ a_j \}$ indexed by the summands $\phi_i$ in $\phi$. 
We define $z_\phi \in S_{\phi}$ as\[
 z_{\phi} = (m_j a_j) \in \prod_j  (\ZZ / 2\ZZ) a_j.
\]
and call it the central element of $S_{\phi}$.
\subsection{The LLC for unitary group}

In this subsection, we introduce the LLC for unitary groups and state some of its properties which we need later.

Let $V^+$ and $V^-$ be the $n$-dimensional $\varepsilon$-Hermitian spaces with $\epsilon(V^+) = +1$, $\epsilon(V^-) = -1$ respectively.  Let $\Irr(\U(V^{\bullet}))$ be the set of  equivalence classes of irreducible smooth representations of $\U(V^{\bullet})$.

\noindent For an $L$-parameter $\phi$ of $\U(V^{\pm})$, there is an associated finite subset $\Pi_{\phi} \subset \Irr(\U(V^{\pm}))$, so called the Vogan $L$-packet satisfying
\[  \Irr(\U(V^+)) \sqcup \Irr(\U(V^-))  = \bigsqcup_{\phi}  \Pi_{\phi}. \]
(Here, $\phi$ on the right-hand side runs over all equivalence classes of $L$-parameters for $\U(V^\pm)$.)

\noindent For each $\epsilon = \pm 1$, we denote the set of irreducible representations of $\U(V^{\epsilon})$ in $ \Pi_{\phi}$ by $ \Pi_{\phi}^{\epsilon}$. Then we have a decomposition as follows:
\[  \Pi_{\phi}  = \Pi_{\phi}^+\sqcup   \Pi_{\phi}^-. \]

As explained in \cite[\S 12]{Gan2}, if we fix an additive character of $\psi:F^{\times} \to \CC$, there is a bijection $$J^{\psi}(\phi):\Pi_{\phi} \to \Irr( S_{\phi}),$$
where $\Irr( S_{\phi})$ is the set of irreducible characters of $S_{\phi}$.

When $n$ is odd, this bijection is canonical but when $n$ is even, it depends on the choice of $\psi$. More precisely, it is determined by the $\N_{E/F}(E^{\times})$-orbit of nontrivial additive characters 
\[
\begin{cases}
 \psi^E:E/F \rightarrow \CC^{\times} & \text{if $\varepsilon = +1$;} \\
 \psi:F \rightarrow \CC^{\times} & \text{if $\varepsilon = -1$.}
\end{cases}
\] 
where \[  \psi^E(x) : = \psi(\tfrac{1}{2} \Tr_{E/F}(\delta x)). \] 
So, it is reasonable to separate $J^{\psi}$ using the notations $J_{\psi^E}$ and  $J_{\psi}$ according to $\varepsilon = +1$ or $\varepsilon = -1$. However, when $n$ is even, we write
\[J^{\psi}=
\begin{cases}
 J_{\psi^E} & \text{if $\varepsilon = +1$;} \\
 J_{\psi} & \text{if $\varepsilon = -1$},
\end{cases} 
\] for convenience and even when $n$ is odd, we use the same notation $J^{\psi}$ for the canonical bijection.\\

Hereafter, we fix an additive character of $\psi:F^{\times} \to \CC$ once and for all and when it comes to the LLC of unitary groups, we shall use a bijection $$J^{\psi}(\phi):\Pi_{\phi} \to \Irr( S_{\phi})$$ for all $L$-parameter $\phi$ of $U(V^{\pm})$ as above. 

Using these fixed bijections, all irreducible smooth representations of $\U(V^{\pm})$ can be labelled as $\pi(\phi,\eta)$ for some unique pair of $L$-parameter $\phi$ of $U(V^{\pm})$ and $\eta \in \Irr( S_{\phi})$.

Lastly, we briefly list some properties of the LLC for unitary group.

\begin{itemize}
\item
$\pi(\phi,\eta)$ is a representation of $\U(V^{\epsilon})$ if and only if $\eta(z_\phi)  = \epsilon$.

\item $\pi(\phi,\eta)$ is a tempereded representation if and only if $\phi$ is tempered.

\item $\pi(\phi,\eta)$ is a discrete series representation if and only if $\phi$ is discrete.

\item There is the canonical identification between the component groups $S_{\phi}$ and $S_{\phi^{\vee}}$. Under such identification, when $\pi$ is $\pi(\phi,\eta)$, its contragradient representation $\pi^{\vee}$ is $\pi(\phi^{\vee},\eta\cdot \nu)$ where \[  
\nu(a_j)  = \begin{cases}
\omega_{E/F}(-1)^{\dim \phi_j} & \text{if $\dim_{\CC} \phi$ is even;}  \\
1 & \text{if $\dim_{\CC} \phi$ is odd.} \end{cases} \]
(The last property follows from a result of Kaletha \cite[Theorem 4.9]{kel}.)
\end{itemize}

\section{\textbf{Local theta correspondence}}

In this section, we state the local theta correspondence and some of its properties for two pairs of unitary groups, namely, $(U(n),U(n+1))$, $(U(n),U(n+2))$. From now on, we shall distinguish the notations for hermitian and skew hermitian spaces. Namely, for $\epsilon = \pm 1$, we denote the $n$-dimensional Hermitian space with $\epsilon(V_n^\epsilon) = \epsilon$ by $V_n^\epsilon$ and  the $n$-dimensional skew-Hermitian space with $\epsilon(W_n^\epsilon) = \epsilon$ by $W_n^\epsilon$.

\subsection{The Weil representation for Unitary groups}
\label{SS:Weil}In this subsection, we introduce the Weil representation associated to the reductive dual pair $(U(V), U(W))$.

Given a Hermitian and a skew-Hermitian spaces $(V,\langle,\rangle_{V})$ and $(W,\langle,\rangle_{W})$  over $E$ respectively, we define the symplectic space $\mathbb{W}_{V,W}$ over $F$ as follows:
$$
\mathbb{W}_{V,W} := \operatorname{Res}_{E/F} (V \otimes_E W)
$$
with the symplectic form $$
\langle v \otimes w,v' \otimes w' \rangle_{\mathbb{W}_{V,W}} := \operatorname{tr}_{E/F}\left(\langle v,v'\rangle_{V}\otimes {\langle w,w' \rangle}_{W}\right).
$$
We also consider the associated symplectic group $Sp(\mathbb{W}_{V,W})$ preserving $\langle \cdot,\cdot \rangle_{\mathbb{W}_{V,W}}$. Then there is a natural map $$U(V) \times U(W)\longrightarrow  Sp(\mathbb{W}_{V,W}).$$ Note that the metaplectic $\mathbb{C}^{1}$-cover  $Mp(\mathbb{W}_{V,W})$  satisfies the following short exact sequence :

$$1\to \CC^{1} \to Mp(\mathbb{W}_{V,W})\to Sp(\mathbb{W}_{V,W})\to 1.$$\\ 
Let $\omega_\psi$ be the Weil representation of $Mp(\mathbb{W}_{V,W})$ with respect to an additive character $\psi:F\to\CC^\times$.

\noindent We choose a pair of unitary characters $(\chi_V,\chi_W)$ of  $E^{\times}$ such that 
$$\chi_V|_{F^{\times}}:=\omega_{E/F}^{\text{dim}_E V} \quad \text{ and } \quad \chi_W|_{F^{\times}}:=\omega_{E/F}^{\text{dim}_E W}.$$
(It is always possible to choose such a pair of characters. For example, fix an unitary character $\chi$ of $E^{\times}$ whose restriction to $F^{\times}$ is $\omega_{E/F}$ and take $\chi_V=\chi^{\text{dim}_E V}$ and $\chi_W=\chi^{\text{dim}_E W}$.)

Then by \cite[\S 1]{Ha}, such a choice $(\chi_{V}, \chi_{W})$ determines a splitting homomorphism $$
\iota_{\chi_{V}, \chi_{W}}:U(V)\times U(W)\to Mp(\mathbb{W}_{V,W})$$ 
and so we have  a Weil representation $\omega_\psi\circ \iota_{\chi_{V},\chi_{W}}$ of $U(V) \times U(W)$.

Throughout the rest of the paper, we denote briefly $\omega_\psi\circ \iota_{\chi_{V},\chi_{W}}$ by $\omega_{\psi,V,W}$ when the choice of $(\chi_V,\chi_W)$ is clear from the context.


\subsection{Local theta correspondence} In this subsection, we state some properties of the local theta correspondence for a pair of unitary groups $(U(V),U(W))$.

Let $\omega_{\psi,W,V}$ be a Weil representation of $U(V) \times U(W)$ and $\pi$ be an irreducible smooth representation of $U(W)$. Then there exists some finite length smooth representation $\Theta_{\psi,V,W}(\pi)$ of $U(V)$ such that \[ \Theta_{\psi,V,W}(\pi) \boxtimes \pi \] is the maximal $\pi$-isotypic quotient of $\omega_{\psi,V,W}$. We write $\theta_{\psi,V,W}(\pi)$ the the maximal semisimple quotient of $\Theta_{\psi,V,W}(\pi)$.
The Howe duality, which is proven for $p \ne 2$ by Waldspurger \cite{w} and by Gan and Takeda \cite{gt1}, \cite{gt2} in general, says that  $\theta_{\psi,V,W}(\pi)$ is either zero or irreducible. It is also known by Kudla \cite{ku} that if $\pi$ is supercuspidal, then  $\Theta_{\psi,V,W}(\pi)$ is zero or irreducible (and thus is equal to $\theta_{\psi,V,W}(\pi)$).

\noindent The local theta correspondence (LTC) describes the relationship between $\pi$ and $\Theta_{\psi,V,W}(\pi)$ (or $\theta_{\psi,V,W}(\pi)$). For the cases $|\text{dim} V - \text{dim} W|\le 1$, D.~Prasad \cite{pra} conjectured the LTC in terms of the LLC and Gan and Ichino proved it in \cite{gi}, \cite{iw}. For the cases $|\text{dim} V - \text{dim} W|\ge 2$, Atobe and Gan \cite{ag} established the precise LTC when the representation $\pi$ of $U(W)$ is tempered.

Since we shall use only two kinds of LTC for $|\text{dim} V - \text{dim} W|=1$ and $|\text{dim} V - \text{dim} W|=2$, we shall elaborate on these two LTC separately. 
Before going on, we fix an additive character $\psi:F \to \CC^{\times}$. 
\subsection{Case (i)}

We first consider the theta correspondence for $\U(V_{n+1}) \times \U(W_n)$. The following summarizes some results of  \cite{ag},\cite{gi}, \cite{iw}.

\begin{thm}  \label{Tae}

Let $\phi$ be an $L$-parameter for $\U(W_n^{\pm})$. Write $$\phi=m_1\phi_1+\cdots +m_r\phi_r +\phi' + {}^c\phi'^{\vee},$$ where $\phi_1,\dots, \phi_r$ are distinct irreducible conjugate self-dual representations of $\WD_E$ with the same type as $\phi$ and $\phi'$ is a sum of irreducible representations of $\WD_E$ which are not conjugate self-dual with the same type as $\phi$. Then we have the following:

\begin{enumerate}

\item Suppose that $\phi$ does not contain $\chi_{V_{n+1}}$.

\begin{enumerate}

 \item For any $\epsilon,\epsilon' \in \{\pm 1\}$ and $\pi \in \Pi_{\phi}^{\epsilon'}$, 
 $\Theta_{\psi, V_{n+1}^{\epsilon}, W_n^{\epsilon'}}(\pi)$ is nonzero and the $L$-parameter of $\theta_{\psi, V_{n+1}^{\epsilon}, W_n^{\epsilon'}}(\pi)$ is 
 \[  \theta(\phi) = (\phi \otimes \chi_{V_{n+1}}^{-1}\chi
_{W_n}) \oplus  \chi_{W_n}. \]

 \item For each $\epsilon = \pm 1$, the theta correspondence map $\pi \mapsto \theta_{\psi, V_{n+1}^{\epsilon}, W_n^{\pm}}(\pi)$  is a bijection between
\[  \Pi_{\phi} \longleftrightarrow \Pi_{\theta(\phi)}^{\epsilon}.  \]  

\item  Write 
\[  S_{\phi}=\prod_i(\ZZ/2\ZZ)a_i \ , \ S_{\theta(\phi)}  =\big( \prod_i(\ZZ/2\ZZ)a_i \big)\times (\ZZ /2 \ZZ) b_1. \]
(Here, the component $(\ZZ/2\ZZ)b_1$ in $S_{\theta(\phi)}$ comes from the summand $\chi_{W_n}$ in $\theta(\phi)$.) \\
Then  the two bijections from LLC \[  J^{\psi}(\phi):  \Pi_{\phi} \longleftrightarrow \Irr(S_{\phi})   \quad \text{and} \quad J^{\psi}(\theta(\phi)):  \Pi_{\theta(\phi)} \longleftrightarrow \Irr(S_{\theta(\phi)}), \]
yields a bijection
\begin{align*}
 \Irr(S_{\phi}) & \longleftrightarrow \Irr^{\epsilon}(S_{\theta(\phi)}) \\
 \eta & \longleftrightarrow \theta(\eta)
\end{align*}
induced by the theta correspondence, where $\Irr^{\epsilon}(S_{\theta(\phi)})$ is the set of irreducible characters $\eta'$ of $S_{\theta(\phi)}$ such that $\eta'(z_{\theta(\phi)}) = \epsilon$. Furthermore, the $\eta$  and $\theta(\eta)$ are related as follows:  
$$\theta(\eta)|_{S_{\phi}}  =\eta.$$

\end{enumerate}

\item Suppose that $\phi$ contains $\chi_V{_{n+1}}$ and fix $\epsilon' \in \{\pm 1\}$. 

\begin{enumerate}

\item For any $\pi \in \Pi_{\phi}^{\epsilon'}$, exactly one of $\Theta_{\psi, V_{n+1}^+, W_n^{\epsilon'}}(\pi)$ or $\Theta_{\psi, V_{n+1}^-, W_n^{\epsilon'}}(\pi)$ is nonzero.

\item If $\Theta_{\psi, V_{n+1}^{\epsilon}, W_n^{\epsilon'}}(\pi)$ is nonzero for some $\epsilon$, then the $L$-parameter of $\theta_{\psi, V_{n+1}^{\epsilon}, W_n^{\epsilon'}}(\pi)$ is
 \[  \theta(\phi) = (\phi \otimes \chi_{V_{n+1}}^{-1} \chi_{W_n} ) \oplus  \chi_{W_n}. \]

\item The theta correspondence map $\pi \mapsto \theta_{\psi, V_{n+1}^{\epsilon}, W_n^{\epsilon'}}(\pi)$ yields a bijection between
 \[  \Pi_{\phi} \longleftrightarrow \Pi_{\theta(\phi)}. \]
 
 \item We can identify $S_{\phi}$ and $S_{\theta(\phi)}$. Under such identification,  the theta correspondence induces a bijection 
\begin{align*}
  \Irr(S_{\phi}) & \longleftrightarrow \Irr(S_{\theta(\phi)}) \\
  \eta & \longleftrightarrow \theta(\eta)
\end{align*}

and $\theta(\eta)=\eta$.
\end{enumerate} 

\item Suppose $\phi$ is tempered. If $\Theta_{\psi, V_{n+1}^{\epsilon}, W_n^{\epsilon'}}(\pi)$ is nonzero for some $\pi \in \Pi_{\phi}^{\epsilon'}$, then $\Theta_{\psi, V_{n+1}^{\epsilon}, W_n^{\epsilon'}}(\pi)$ is irreducible  (and thus equal to $\theta_{\psi, V_{n+1}^{\epsilon}, W_n^{\epsilon'}}(\pi).)$ 

\item If $\phi$ is tempered and $\Theta_{\psi, V_{n+1}^{\epsilon}, W_n^{\epsilon'}}(\pi)$ is nonzero for some $\pi \in \Pi_{\phi}^{\epsilon'}$, then $$\theta_{\psi, W_n^{\epsilon'},V_{n+1}^{\epsilon}}\big(\Theta_{\psi, V_{n+1}^{\epsilon}, W_n^{\epsilon'}}(\pi)\big) \cong \pi.$$ Similarly, if $\phi_1$ is tempered $L$-parameter of $U(V_{n+1}^{\pm})$ and $\Theta_{\psi, W_n^{\epsilon'},V_{n+1}^{\epsilon}}(\pi_1)$ is nonzero for $\pi_1\in \Pi_{\phi_1}^{\epsilon'}$, then $$\Theta_{\psi, V_{n+1}^{\epsilon'},W_{n}^{\epsilon}}\big(\theta_{\psi, W_{n}^{\epsilon}, V_{n+1}^{\epsilon'}}(\pi_1)\big) \cong \pi_1.$$ 

\end{enumerate}

\end{thm}

\begin{proof}Except for (iv), other properties are from Theorem C.5 in \cite{gi} and \cite{iw}.  So we shall elaborate on (iv). Since $\theta(\phi)$ contains $\chi_{W_n}$, we know that $\Theta_{\psi, W_n^{\epsilon'},V_{n+1}^{\epsilon}}\big(\Theta_{\psi, V_{n+1}^{\epsilon}, W_n^{\epsilon'}}(\pi)\big)$ is nonzero by the above property (iii) and Theorem 4.1 in \cite{ag}.
Note that $\Theta_{\psi, W_n^{\epsilon'},V_{n+1}^{\epsilon}}\big(\Theta_{\psi, V_{n+1}^{\epsilon}, W_n^{\epsilon'}}(\pi)\big)\boxtimes \Theta_{\psi, V_{n+1}^{\epsilon}, W_n^{\epsilon'}}(\pi)$ is the maximal $\Theta_{\psi, V_{n+1}^{\epsilon}, W_n^{\epsilon'}}(\pi)$-isotypic quotient of $ \omega_{\psi,V_{n+1},W_n}$ and $ \pi \boxtimes \Theta_{\psi, V_{n+1}^{\epsilon}, W_n^{\epsilon'}}(\pi)$ is a $\Theta_{\psi, V_{n+1}^{\epsilon}, W_n^{\epsilon'}}(\pi)$-isotypic quotient of $ \omega_{\psi,V_{n+1},W_n}$. Thus $\pi$ should be a irreducible quotient of $\Theta_{\psi, W_n^{\epsilon'},V_{n+1}^{\epsilon}}\big(\Theta_{\psi, V_{n+1}^{\epsilon}, W_n^{\epsilon'}}(\pi)\big)$. By Howe duality, $\Theta_{\psi, W_n^{\epsilon'},V_{n+1}^{\epsilon}}\big(\Theta_{\psi, V_{n+1}^{\epsilon}, W_n^{\epsilon'}}(\pi)\big)$ has a unique irreducible quotient $\theta_{\psi, W_n^{\epsilon'},V_{n+1}^{\epsilon}}\big(\Theta_{\psi, V_{n+1}^{\epsilon}, W_n^{\epsilon'}}(\pi)\big)$ and thus it is isomorphic to $\pi$. \\On the other hand, since $\Theta_{\psi, W_{n}^{\epsilon}, V_{n+1}^{\epsilon'}}(\pi_1)$ is nonzero, the $L$-parameter of $\theta_{\psi, W_{n}^{\epsilon}, V_{n+1}^{\epsilon'}}(\pi_1)$ is tempered by Proposition 5.5 in \cite{ag}.
\noindent Since $\Theta_{\psi, V_{n+1}^{\epsilon'},W_{n}^{\epsilon}}\big(\theta_{\psi, W_{n}^{\epsilon}, V_{n+1}^{\epsilon'}}(\pi_1)\big) \boxtimes \theta_{\psi, W_{n}^{\epsilon}, V_{n+1}^{\epsilon'}}(\pi_1)$ is the maximal $ \theta_{\psi, W_{n}^{\epsilon}, V_{n+1}^{\epsilon'}}(\pi_1)$-isotypic quotient of $ \omega_{\psi,W_n,V_{n+1}}$ and $\pi_1 \boxtimes \theta_{\psi, W_{n}^{\epsilon}, V_{n+1}^{\epsilon'}}(\pi_1)$ is the $\theta_{\psi, W_{n}^{\epsilon}, V_{n+1}^{\epsilon'}}(\pi_1)$-isotypic quotient of $ \omega_{\psi,W_n,V_{n+1}}$, $\pi_1$ should be a quotient of $\Theta_{\psi, V_{n+1}^{\epsilon'},W_{n}^{\epsilon}}\big(\theta_{\psi, W_{n}^{\epsilon}, V_{n+1}^{\epsilon'}}(\pi_1)\big)$. From property (iii) in the above, we know that $\Theta_{\psi, V_{n+1}^{\epsilon'},W_{n}^{\epsilon}}\big(\theta_{\psi, W_{n}^{\epsilon}, V_{n+1}^{\epsilon'}}(\pi_1)\big)$ is irreducible and so $\Theta_{\psi, V_{n+1}^{\epsilon'},W_{n}^{\epsilon}}\big(\theta_{\psi, W_{n}^{\epsilon}, V_{n+1}^{\epsilon'}}(\pi_1)\big) \cong \pi_1.$ 

\end{proof}

\subsection{Case (ii)} Now we shall consider the theta correspondence for $\U(V_{n+2}^\epsilon) \times \U(W_n^{\epsilon'})$.
\begin{thm}\label{theta}Let $\phi$ be an $L$-parameter of $U(W_n^{\pm})$. Assume that $\Pi_{\phi}$ consists of only supercuspidal representations. For any fixed  $\epsilon\in\{\pm1\}$ and let $\epsilon'\ =\epsilon \cdot \epsilon(\half,\phi \otimes \chi_{V_{n+2}}^{-1},\psi_2^E)$. Then we have:

\begin{enumerate}
\item For any $\pi \in \Pi_{\phi}^{\epsilon'}$, $\Theta_{\psi, V_{n+2}^{\epsilon},W_n^{\epsilon'}}(\pi)$ is nonzero and irreducible.

\item The $L$-parameter $\theta(\phi)$ of $\Theta_{\psi, V_{n+2}^{\epsilon},W_n^{\epsilon'}}(\pi)$ is $\theta(\phi)=(\phi\otimes \chi_{V_{n+2}}^{-1}\chi_{W_n}) \oplus  \left(\chi_{W_n}|\cdot|_E^{\half} \oplus \chi_{W_n}|\cdot|_E^{-\half} \right).$

\item The theta correspondence $\pi \mapsto \theta_{\psi, V_{n+2}^{\epsilon}, W_n^{\epsilon'}}(\pi)$ induces a bijection
\[  \Pi_{\phi} \longleftrightarrow \Pi_{\theta(\phi)}.  \]  
\end{enumerate}

Since $\phi$ is discrete,  we can write $\phi=\phi_1+\cdots +\phi_r $, where $\phi_1,\dots, \phi_r$ are distinct irreducible conjugate self-dual representations of $\WD_E$ with the same type as $\phi$. Let $S_{\phi}\ = S_{\theta(\phi)}\ = \prod_{j=1}^r (\ZZ/2\ZZ)c_j$.

Using the following two bijections from the LLC
\begin{itemize}
\item $J^{\psi}(\phi): \Pi_{\phi} \longleftrightarrow \Irr(S_{\phi}) $
\item $J^{\psi}(\theta(\phi)): \Pi_{\theta(\phi)} \longleftrightarrow \Irr(S_{\theta(\phi)}),$
\end{itemize}
we obtain a bijection induced from the theta correspondence 
\begin{align*}
 \Irr(S_{\phi}) & \longleftrightarrow \Irr(S_{\theta(\phi)}) \\
 \eta & \longleftrightarrow \theta(\eta).
\end{align*}  
Furthermore, the bijection is explicated as follows:
\begin{equation}\label{theta1} \theta(\eta)(c_j)=\eta(c_j) \cdot \epsilon(\frac{1}{2},\phi^{(j)} \otimes \chi_{V_{n+2}}^{-1},\psi_{2}^E ),\end{equation}

\begin{proof} With our choice of $\epsilon'$, the non-vanishing of $\Theta_{\psi, V_{n+2}^{\epsilon},W_n^{\epsilon'}}(\pi)$ follows from Theorem 4.1 (2) in \cite{ag}. Since $\Pi_{\phi}^{\epsilon'}$ consists of supercuspidal representations of $U(W_n^{\epsilon'})$, their theta lifts $\Theta_{\psi, V_{n+2}^{\epsilon},W_n^{\epsilon'}}(\pi)$ are irreducible. The properties (ii), (iii) follow from Theorem 4.3 (4) and Theorem 6.5 in \cite{ag}.
\end{proof}

\end{thm}

\begin{rem}\label{rem4}There is an $L$-parameter which satisfies our assumption. For example, if $E/F$ is an unramified extension, DeBacker and Reeder \cite[\S 12]{dr} defined the notion of tamely regular semisimple elliptic Langlands parameter and showed that its $L$-packet consists of depth-zero supercuspidal representations. 

\noindent For symplectic and special orthogonal groups, the condition for $\phi$ to be supercuspidal in the above sense is known by Moeglin and Xu [\cite{Mo0}, \cite{Xu}]. In view of its criterion, we can make the following conjecture for the unitary groups.

\begin{con}For an $L$-parameter $\phi:\WD_F \rightarrow \GL_n(\CC)$ of $\U(V^{\pm})$, its $L$-packet $\Pi_{\phi}$ consists of supercuspidal representations if and only if $\phi$ is discrete and its restriction to $\SL_2(\CC)$ is trivial.

\end{con}
\end{rem}

\section{\textbf{Main Theorem}}

We prove our main theorem in this section. Since our main theorem is based upon the results of Plessis and Gan-Ichino, we first elaborate on their results on the GGP conjecture for unitary groups and then we state our main theorem. In  (\cite[\S3]{iw}), Gan and Ichino have made a excellent exposition on the GGP conjecture for both (B)  and (FJ), we shall quote their treatment here. Throughout this section, we fix a nontrivial additive character $\psi:F \to \CC^{\times}$ and make a tacit use of the associated bijection $J^{\psi}$ appearing in the LLC.
\subsection{Pairs of spaces}

To explain the (B) and (FJ) cases of the GGP conjecture simulataneously, we first consider the pair of spaces:
\[ \begin{cases} V_n^{\epsilon} \subset V_{n+1}^{\epsilon} \quad \quad \text{(\textbf{Bessel case})} \\   W_n^{\epsilon} = W_n^{\epsilon}. \quad  \quad \text{ (\textbf{Fourier-Jacobi case})}.
\end{cases} \]
\noindent For $a \in F^{\times}$, we denote a 1-dimensional Hermitian space with form $a \cdot \N_{E/F}$ by $L_a$. Then
\[  V_{n+1}^{\epsilon}/V_n^{\epsilon}   \cong L_{(-1)^n}. \]
Write \[  G_n^\epsilon =  \begin{cases} \U(V_{n+1}^\epsilon) \times \U(V_{n}^\epsilon) \quad \quad \text{(\textbf{Bessel case})} \\  \U(W_n^\epsilon) \times \U(W_n^\epsilon) \quad  \quad \text{ (\textbf{Fourier-Jacobi case})}, \end{cases} H_n^\epsilon =  \begin{cases}  \U(V_n^\epsilon) \quad \quad \text{  \textbf{(Bessel case)}} \\  \U(W_n^\epsilon) \quad  \quad \text{ (\textbf{Fourier-Jacobi case})}. \end{cases}\]

\noindent In both cases, we have a diagonal embedding
\[ \Delta : H_n^\epsilon \hookrightarrow G_n^\epsilon.\]

\noindent For an $L$-parameter $\phi = \phi^{\diamondsuit} \times \phi^{\heartsuit}$ for $G_n^\pm$, we denote its associated component group by
\[  S_{\phi}  = S_{\phi^{\diamondsuit}}  \times S_{\phi^{\heartsuit}}\]
and the set of irreducible characters of $S_{\phi}$ by $\Irr (S_{\phi})=\Irr (S_{\phi}) \times \Irr (S_{\phi})$. 

Note that for an $\eta \in \Irr (S_{\phi})$, its corresponding representation $\pi(\eta) \in \Pi_{\phi}=\Pi_{\phi^{\diamondsuit}} \times \Pi_{\phi^{\heartsuit}}$ (under the LLC) is a representation of $G_n^{\epsilon}$ if and only if 
\[  \eta(z_{\phi^{\diamondsuit}}, \textbf{0})  = \eta(\textbf{0},z_{\phi^{\heartsuit}})  = \epsilon, \]
where \textbf{0} denote the identity element in both $S_{\phi^{\diamondsuit}}$ and $S_{\phi^{\heartsuit}}$.

\subsection{The recipe} 
\label{SS:eta}In this subsection, we shall describe the recipe of the GGP conjecture for both the (B) and (FJ). For an $L$-parameter $\phi =  \phi^{\diamondsuit} \times \phi^{\heartsuit}$ of $G_n^{\epsilon}$, write \[  S_{\phi^{\diamondsuit}}  = \prod_i  (\ZZ /2 \ZZ) a_i \quad \text{and} \quad S_{\phi^{\heartsuit}} = \prod_j  (\ZZ / 2\ZZ) b_j. \]
Note that $\eta \in \Irr (S_{\phi})$ is completely determined by the values $\eta(a_i, \textbf{0})\in \{ \pm 1\}$ and $\eta(\textbf{0},b_j) \in \{ \pm 1\}$. For convenience, we simply write $\eta(a_i, \textbf{0}) $ by $\eta(a_i)$ and $\eta( \textbf{0},b_j)$ by $\eta(b_j)$. 

Now, we define the distinguished characters of $S_{\phi}$ for the (B) and (FJ) cases as follows:

\begin{enumerate}

\item \textbf{(Bessel)}
Set $\psi^E_{-2}(x)  = \psi(-\Tr_{E/F}(\delta x))$. We define $\eta^{\spadesuit}\in \text{Irr}(S_{\phi})$ as follows:
\[
\begin{cases}
 \eta^{\spadesuit}(a_i)  = \epsilon(\frac{1}{2}, \phi^{\diamondsuit}_i \otimes \phi^{\heartsuit}, \psi^E_{-2}); \\
 \eta^{\spadesuit}(b_j)  = \epsilon(\frac{1}{2},  \phi^{\diamondsuit} \otimes \phi^{\heartsuit}_{j}, \psi^E_{-2}).
\end{cases}
\]
\\
\item \textbf{(Fourier-Jacobi)} Set $\psi^E(x)  = \psi(\frac{1}{2}\Tr_{E/F}(\delta x))$ and $\psi^E_2(x)  = \psi(\Tr_{E/F}(\delta x))$. The distinguished character $\eta^{\clubsuit}$ of $S_{\phi}$ depends on the parity of $n=\dim W_n$.\\

\begin{itemize}
\item If $n$ is odd, $\eta^{\clubsuit}$ is defined as
\[ \begin{cases}
 \eta^{\clubsuit}(a_i) = \epsilon(\frac{1}{2}, \phi^{\diamondsuit}_{i} \otimes \phi^{\heartsuit} \otimes \chi^{-1},  \psi_2^E);  \\
  \eta^{\clubsuit}(b_j)  = \epsilon(\frac{1}{2}, \phi^{\diamondsuit} \otimes \phi^{\heartsuit}_{j} \otimes \chi^{-1}, \psi_2^E). \end{cases} \]

\item If $n$ is even, $\eta^{\clubsuit}$ is defined as 
\[  \begin{cases}
\eta^{\clubsuit}(a_i) = \epsilon(\frac{1}{2}, \phi^{\diamondsuit}_i \otimes \phi^{\heartsuit} \otimes \chi^{-1},  \psi^E); \\
  \eta^{\clubsuit}(b_j)  = \epsilon(\frac{1}{2}, \phi^{\diamondsuit} \otimes \phi^{\heartsuit}_{j} \otimes \chi^{-1}, \psi^E). \end{cases} \]
\end{itemize}
\end{enumerate}

\subsection{Theorem (B) and (FJ) for generic parameter}\label{gen}

We state the results of Plessis(\cite{bp1}, \cite{bp2}, \cite{bp3}) and Gan-Ichino(\cite{iw}) on the GGP conjecture.

\begin{itemize}

\item[(B)$_n$]
For a \emph{generic} $L$-parameter $\phi$ for $G_n^\pm = \U(V_{n+1}^\pm) \times \U(V_{n}^\pm)$ and a representation $\pi(\eta) \in \Pi_{\phi}$ of $G_n^{\pm}$,
\[  \Hom_{\Delta H_n^{\pm}}(\pi(\eta), \CC) \ne 0 \Longleftrightarrow  \eta = \eta^{\spadesuit}. \]

\item[(FJ)$_n$]
For a \emph{generic} $L$-parameter $\phi$ for $G_n^\pm = \U(W_n^\pm) \times \U(W_n^\pm)$ and a representation $\pi(\eta) \in \Pi_{\phi}$ of $G_n^{\pm}$, 
\[ \Hom_{\Delta H_n^{\pm}}(\pi(\eta), \omega_{\psi,\chi, W_n}) \ne 0 \Longleftrightarrow \eta = \eta^{\clubsuit}. \]

\end{itemize}

This paper investigates (B)$_{n+1}$ of the GGP conjecture for some non-generic $L$-parameters of $G_{n+1}^\pm$. Let us state our main theorem in the following.

\begin{thm}\label{thm2} Let $E/F$ be a quadratic extension of local fields of characteristic zero and the unitary groups we are considering here are all associated to this extension. For an integer $n\ge 1$, we choose a pair of unitary characters $(\chi_{V},\chi_{W})$ of  $E^{\times}$ such that 
$$\chi_{V}|_{F^{\times}}:=\omega_{E/F}^{n+2} \quad \text{ and } \quad \chi_{W}|_{F^{\times}}:=\omega_{E/F}^{n}.$$ 

\noindent Let $\phi_1$ be a $L$-parameter of $U(W_n^{\pm})$ and we assume that $\Pi_{\phi_1}$ consists of supercuspidal representations. (See Remark \ref{rem4}) Define a non-generic $L$-parameter $\theta(\phi_1)$ of $U(V_{n+2}^{\pm})$ by
\[\theta(\phi_1)=\phi_1\otimes \chi_{V}^{-1}\chi_{W}\oplus  \left( \chi_{W}|\cdot|_E^{\half} \oplus \chi_{W}|\cdot|_E^{-\half} \right).\]
Let $\phi$ be a tempered $L$-parameter of $U(V_{n+1}^{\pm})$. Then we have the followings:
 \begin{itemize}
\item If $\phi$ does not contain $\chi_{W}$, then  \beqna \label{non}{\sum_{(\pi_{n+2}, \pi_{n+1})\in \Pi^{\pm}_{\theta(\phi_1)} \times \Pi^{\pm}_{\phi}}\dim_{\CC}\Hom_{U(V_{n+1}^{\pm})}(\pi_{n+2},\pi_{n+1})}=0.\eeqna \\

\item  If $\phi$ contains $\chi_{W}$, we can write  \[ \phi=\theta(\phi_2)=\phi_2 \otimes \chi_{V}^{-1} \chi^{(-1)^{n}}\chi_{W}\oplus \chi_{W} \] for some temperd $L$-parameter $\phi_2$ of $U(W_n^{\pm})$. Then

 \beqna \label{exi}{\sum_{(\pi_{n+2}, \pi_{n+1})\in \Pi^{\pm}_{\theta(\phi_1)} \times \Pi^{\pm}_{\theta(\phi_2)}}\dim_{\CC}\Hom_{U(V_{n+1}^{\pm})}(\pi_{n+2},\pi_{n+1})}\ge 1.\eeqna
Furthermore, if $\phi_2$ does not contain $\chi_{V}\chi^{(-1)^{n+1}}$, then\beqna  \label{uni}\sum_{(\pi_{n+2}, \pi_{n+1})\in \Pi^{\pm}_{\theta(\phi_1)} \times \Pi^{\pm}_{\theta(\phi_2)}}\dim_{\CC}\Hom_{U(V_{n+1}^{\pm})}(\pi_{n+2},\pi_{n+1}) =1\eeqna and we can explicate $(\pi_{n+2}, \pi_{n+1})\in \Pi^{\pm}_{\theta(\phi_1)} \times \Pi^{\pm}_{\theta(\phi_2)}$ such that $\dim_{\CC}\Hom_{U(V_{n+1}^{\pm})}(\pi_{n+2},\pi_{n+1})=1$ as follows:\\ \\
Decompose$$\phi_1=\phi_1^{(1)}+\cdots +\phi_1^{(r)}  \quad , \quad \phi_2=m_1\phi_2^{(1)}+\cdots +m_s\phi_2^{(s)} +\phi_2' + {}^c\phi_2'^{\vee}$$
and write 
\ $\begin{cases}  S_{\phi_1}=S_{\theta(\phi_1)}= \prod_i  (\ZZ /2 \ZZ) a_i ; \\ S_{\phi_2}= \prod_j  (\ZZ / 2\ZZ) b_j \end{cases}$ and \ \ $S_{\theta(\phi_2)}=\big(\prod_j  (\ZZ / 2\ZZ) b_j \big) \times (\ZZ / 2\ZZ) c_1.$\\

\noindent Then for $(\pi_{n+2},\pi_{n+1})\in \Pi_{\theta(\phi_1)} \times \Pi_{\theta(\phi_2)},$ 
$$\Hom_{U(V_{n+1}^{\epsilon})}(\pi_{n+2},\pi_{n+1})\ne0 \Leftrightarrow (\pi_{n+2},\pi_{n+1})=(\pi_{\theta(\phi_1)}(\eta^{\diamondsuit}),\pi_{\theta(\phi_2)}(\eta^{\heartsuit}))$$ where $(\eta^{\diamondsuit},\eta^{\heartsuit})\in \Irr(S_{\theta(\phi_1)}) \times  \Irr(S_{\theta(\phi_2)})$ the pair of characters of the component group is specified as follows;\\

\noindent If $n$ is odd,
\begin{equation}\label{etao}\begin{cases}\eta^{\diamondsuit}(a_i)=\epsilon(\half, \phi_1^{(i)}\cdot  \chi_{V}^{-1}\chi_W\otimes \phi^{\vee} ,\psi_2^E),\\ \eta^{\heartsuit}(b_j)=\epsilon(\frac{1}{2},\phi_1 \otimes  (\phi_2^{(j)})^{\vee} \otimes \chi^{-1} ,\psi_2^E),\\ \eta^{\heartsuit}(c_1)=\epsilon(\half, \phi_1\cdot  \chi_{V}^{-1}\chi_W\otimes \phi^{\vee} ,\psi_2^E) \cdot \epsilon(\frac{1}{2},\phi_1 \otimes  (\bar{\phi_2})^{\vee} \otimes \chi^{-1} ,\psi_2^E).  \end{cases}\end{equation}  

\noindent If $n$ is even,
\begin{equation}\label{etae}\begin{cases}\eta^{\diamondsuit}(a_i)=\epsilon(\half, \phi_1^{(i)}\cdot  \chi_{V}^{-1}\chi_W\otimes \phi^{\vee} ,\psi^E),\\ \eta^{\heartsuit}(b_j)=\epsilon(\frac{1}{2},\phi_1 \otimes  (\phi_2^{(j)})^{\vee} \otimes \chi^{-1} ,\psi^E),\\ \eta^{\heartsuit}(c_1)=\epsilon(\half, \phi_1\cdot  \chi_{V}^{-1}\chi_W\otimes \phi^{\vee} ,\psi^E) \cdot \epsilon(\frac{1}{2},\phi_1 \otimes  (\bar{\phi_2})^{\vee} \otimes \chi^{-1} ,\psi^E).  \end{cases}\end{equation}  
where $\bar{\phi_2}=m_1\phi_2^{(1)}+\cdots +m_s\phi_2^{(s)}$
\end{itemize}
\end{thm}

\begin{proof} \noindent  We first prove (\ref{non}). To prove it, suppose that \[ \dim_{\CC}\Hom_{U(V_{n+1}^{\epsilon})}(\pi_{n+2},\pi_{n+1}) \ne 0,  \quad \text{for some $\epsilon \in \{\pm1\}$ and $(\pi_{n+2}, \pi_{n+1})\in \Pi^{\epsilon}_{\theta(\phi_1)} \times \Pi^{\epsilon}_{\phi}.$}\]
\noindent Put $\epsilon'=\epsilon \cdot \epsilon(\half,\phi_1\otimes \chi_{V}^{-1},\psi_2^E)$ and we consider the following see-saw diagram : 

\[
 \xymatrix{
  \U(W_n^{\epsilon'})  \times \U(W_n^{\epsilon'})  \ar@{-}[dr] \ar@{-}[d] & \U(V^{\epsilon}_{n+2}) \ar@{-}[d] \\
  \U(W_n^{\epsilon'}) \ar@{-}[ur] &  \U(V^{\epsilon}_{n+1}) \times \U(L_{(-1)^{n+1}})}
\]
We have three theta correspondence in this diagram : 
\begin{enumerate}
\item $U(V_{n+2}^{\epsilon}) \times U(W_n^{\epsilon})$ relative to the pair of characters $(\chi_W, \chi_V)$;
\item $\U(V^{\epsilon}_{n+1}) \times U(W^{\epsilon'}_n)$ relative to the pair of characters $(\chi_W, \chi_V \cdot \chi^{(-1)^{n}})$;
\item $\U(L_{(-1)^{n+1}}) \times \U(W_n^{\epsilon'})$ relative to the pair of characters $( \chi_W,\chi^{(-1)^{n+1}}).$
\end{enumerate}

By (i) and (iii) of Theorem \ref{theta}, we can write $\pi_{n+2}=\Theta_{\psi,\chi,V_{n+2},W_n}(\sigma)$ for some $\sigma \in \Pi^{\epsilon'}_{\phi_1}$. Then by the see-saw identity, we have 
$$0 \ne \Hom_{U(V_{n+1}^{\epsilon})}(\pi_{n+2},\pi_{n+1}) \simeq \Hom_{U(W_n^{\epsilon'})}(\Theta_{\psi,\chi,W_n,V_{n+1}}(\pi_{n+1}) \otimes  \omega_{\psi,\chi,L_{(-1)^{n+1}}, W_n},\sigma)$$ 
and so we see that $\Theta_{\psi,\chi,W_n,V_{n+1}}(\pi_{n+1}) \ne 0$.
\\Then by Theorem \ref{Tae} (iv), $\pi_{n+1}$ should be a theta lift of some irreducible representation of $U(W_n^{\epsilon'})$, and so by Theorem \ref{Tae}, $\phi$ should contain $\chi_W$. This accounts for (\ref{non}) .

Secondly, we prove (\ref{exi}). 
By (FJ)$_n$, there is some ${\epsilon'}\in \{\pm1\}$ and $(\pi_{\phi_2^{\vee}}, \pi_{\phi_1}) \in \Pi^{\epsilon'}_{\phi^{\vee}_2} \times \Pi^{\epsilon'}_{\phi_1}$ such that $$\Hom_{U(W_n^{\epsilon'})}(\pi_{\phi_{2}^{\vee}}\otimes \pi_{\phi_1},\omega_{\psi,W_n^{\epsilon'}})\ne0.$$
\noindent We shall consider the following see-saw diagram : ($\epsilon$ will be determined soon.)
\[
 \xymatrix{
  \U(W_n^{\epsilon'})  \times \U(W_n^{\epsilon'})  \ar@{-}[dr] \ar@{-}[d] & \U(V^{\epsilon}_{n+2}) \ar@{-}[d] \\
  \U(W_n^{\epsilon'}) \ar@{-}[ur] &  \U(V^{\epsilon}_{n+1}) \times \U(L_{(-1)^{n+1}})}
\]
Note that $ \omega_{\psi,\chi,L_{(-1)^{n+1}}, W_n}=\begin{cases}\omega_{\psi,\chi, W_n} \quad \quad \quad \quad\text{ if $n$ is odd;} \\ \omega^{\vee}_{\psi, \chi, W_n} \quad \quad \quad \quad\text{ if $n$ is even}.\end{cases}$

\noindent Because of the above differences for even and odd $n$, we will deal with even and odd cases separartely.

\noindent \begin{itemize}
\item If $n$ is odd, we put $\epsilon=\epsilon' \cdot \epsilon(\half,\phi_1\otimes \chi_{V}^{-1},\psi_2^E)$ and use the three theta correspondence in the above diagram:
\bigbreak
 \begin{enumerate}
\item $U(V_{n+2}^{\epsilon}) \times U(W_n^{\epsilon})$ relative to the pair of characters $(\chi_W, \chi_V)$;
\item $\U(V^{\epsilon}_{n+1}) \times U(W^{\epsilon'}_n)$ relative to the pair of characters $(\chi_W, \chi_V \cdot \chi^{(-1)^{n}})$;
\item $\U(L_{(-1)^{n+1}}) \times \U(W_n^{\epsilon'})$ relative to the pair of characters $( \chi_W,\chi^{(-1)^{n+1}}).$
\end{enumerate}
\bigbreak

\noindent Since $\pi_{\phi_{2}^{\vee}}, \pi_{\phi_1}$ are both unitary, one has $$\Hom_{U(W_n^{\epsilon'})}\big( (\pi_{\phi^{\vee}_2})^{\vee}\otimes \omega_{\psi,\chi,W_n^{\epsilon'}},\pi_{\phi_1})\ne0.$$
(Note that the $L$-parameter of $(\pi_{\phi_2^{\vee}})^{\vee}$ is $\phi_2$.)\\
Write $\tau:=\Theta_{\psi,V^{\epsilon}_{n+1},W^{\epsilon'}_n}\big((\pi_{\phi^{\vee}_2})^{\vee}\big).$ Then by Theorem \ref{Tae} (i) and (iv), $\tau$ is non-zero and  $\theta_{\psi,W^{\epsilon'}_n,V^{\epsilon}_{n+1}}(\tau)=(\pi_{\phi^{\vee}_2})^{\vee}.$ 
Then $$0 \ne \Hom_{U(W_n^{\epsilon'})}\big( (\pi_{\phi^{\vee}_2})^{\vee}\otimes \omega_{\psi,\chi,W_n^{\epsilon'}},\pi_{\phi_1}) \subseteq \Hom_{U(W_n^{\epsilon'})}(\Theta_{\psi,W^{\epsilon'}_n,V^{\epsilon}_{n+1}}(\tau)\otimes \omega_{\psi,\chi,W_n^{\epsilon'}},\pi_{\phi_1})$$ and so by the see-saw identity, one has $$\Hom_{U(V_{n+1}^{\epsilon})}(\Theta_{\psi,V^{\epsilon}_{n+2},W^{\epsilon'}_{n}}(\pi_{\phi_1}),\tau)\ne 0.$$
Since the $L$-parameter of $\Theta_{V^{\epsilon}_{n+2},W^{\epsilon'}_n}(\pi_{\phi_1})$ is $\theta(\phi_1)$, we proved (\ref{exi}) when $n$ is odd.\\

\item If $n$ is even, we put $\epsilon=\epsilon' \cdot \epsilon(\half,\phi_1^{\vee}\otimes \chi_{V},\psi_2^E)$ and use the three theta correspondence in the above diagram:
\bigbreak
 \begin{enumerate}
\item $U(V_{n+2}^{\epsilon}) \times U(W_n^{\epsilon})$ relative to the pair of characters $(\chi_W^{-1}, \chi_V^{-1})$;
\item $\U(V^{\epsilon}_{n+1}) \times U(W^{\epsilon'}_n)$ relative to the pair of characters $(\chi_W^{-1}, \chi_V^{-1} \cdot \chi^{(-1)^{n}})$;
\item $\U(L_{(-1)^{n+1}}) \times \U(W_n^{\epsilon'})$ relative to the pair of characters $( \chi_W^{-1},\chi^{(-1)^{n+1}}).$
\end{enumerate}
\bigbreak

\noindent Since $\pi_{\phi_1}$ is unitary, $$\Hom_{U(W_n^{\epsilon'})}(\pi_{\phi_{2}^{\vee}}\otimes \omega^{\vee}_{\psi,\chi,W_n^{\epsilon'}},\pi^{\vee}_{\phi_1})\ne0.$$ Write $\tau:=\Theta_{\psi,V^{\epsilon}_{n+1},W^{\epsilon'}_n}\big(\pi_{\phi^{\vee}_2}\big).$ Then by Theorem \ref{Tae} (i) and (iv), $\tau$ is non-zero and $\theta_{\psi,W_n,V_{n+1}}(\tau)=\pi_{\phi^{\vee}_2}.$
Then $$0 \ne \Hom_{U(W_n^{\epsilon'})}\big( \pi_{\phi^{\vee}_2}\otimes \omega^{\vee}_{\psi,\chi,W_n^{\epsilon'}},\pi^{\vee}_{\phi_1}) \subseteq \Hom_{U(W_n^{\epsilon'})}(\Theta_{\psi,\chi,W_n,V_{n+1}}(\tau)\otimes \omega^{\vee}_{\psi,\chi,W_n^{\epsilon'}},\pi^{\vee}_{\phi_1})$$ and so by the see-saw identity, one has $$\Hom_{U(V_{n+1}^{\epsilon})}(\Theta_{\psi,V^{\epsilon}_{n+2},W^{\epsilon'}_{n}}(\pi^{\vee}_{\phi_1}),\tau)\ne 0.$$
Since $\Theta_{\psi,V^{\epsilon}_{n+2},W^{\epsilon'}_{n}}(\pi^{\vee}_{\phi_1})$, $\tau$ are tempered and so unitary, we have
$$\Hom_{U(V_{n+1}^{\epsilon})}\big( \big(\Theta_{\psi,V^{\epsilon}_{n+2},W^{\epsilon'}_{n}}(\pi^{\vee}_{\phi_1})\big)^{\vee},\tau^{\vee}\big)\ne 0$$ and so we proved (\ref{exi}) when $n$ is even.\\
(Note that the $L$-parameters of $\big(\Theta_{\psi,V^{\epsilon}_{n+2},W^{\epsilon'}_{n}}(\pi^{\vee}_{\phi_1})\big)^{\vee}$ and $\tau^{\vee}$ are $\theta(\phi_1)$ and $\theta(\phi_2)$ respectively.)
\end{itemize}

Next, to prove (\ref{uni}), choose some $ (\pi_{n+2},\pi_{n+1})\in \Pi^{\epsilon}_{\theta(\phi_1)} \times \Pi^{\epsilon}_{\theta(\phi_{2})}$ such that  $$\Hom_{U(V_{n+1}^{\epsilon})}(\pi_{n+2},\pi_{n+1})\ne0.$$  Now, we divide the cases according to the parity of $n$ and use the same conjugate self-dual characters for the three theta correspondence as in the proof of (\ref{exi}). 
\begin{itemize} \item If $n$ is odd, put $\epsilon=\epsilon' \cdot \epsilon(\half,\phi_1\otimes \chi_{V}^{-1},\psi_2^E)$. Then by (i) and (iii) of Theorem \ref{theta}, we can write $\pi_{n+2}=\Theta_{\psi,\chi,V_{n+2},W_n}(\sigma)$ for some $\sigma \in \Pi^{\epsilon'}_{\phi_1}$. Using the see-saw identity, one has $$\Hom_{U(V_{n+1}^{\epsilon})}(\pi_{n+2},\pi_{n+1})\cong \Hom_{U(W_n^{\epsilon'})}(\Theta_{\psi, W_n^{\epsilon'},V_{n+1}^{\epsilon}}(\pi_{n+1}) \otimes \omega_{\psi,\chi,W_n^{\epsilon'}},\sigma )\ne 0.$$ In particular, $\Theta_{\psi, W_n^{\epsilon'},V_{n+1}^{\epsilon}}(\pi_{n+1})$ is non-zero and so it is irreducible and tempered by Proposition 5.4 in \cite{ag}. (This is the part where our assumption `$\phi_2$ does not contatin $\chi_{V}\chi^{(-1)^{n+1}}$' is used.)\\
Since  $\Theta_{\psi, W_n^{\epsilon'},V_{n+1}^{\epsilon}}(\pi_{n+1}) $ and $\sigma$ are both unitary, we have $$\Hom_{U(W_n^{\epsilon'})}(\Theta_{\psi, W_n^{\epsilon'},V_{n+1}^{\epsilon}}^{\vee}(\pi_{n+1})\otimes \sigma,\omega_{\psi,\chi,W_n^{\epsilon'}} )\ne0$$ and by Theorem \ref{Tae} (iv), the $L$-parameter of $\Theta_{\psi, W_n^{\epsilon'},V_{n+1}^{\epsilon}}^{\vee}(\pi_{n+1})$ is $\phi_2^{\vee}$. Thus by (FJ)$_n$, we see that $\epsilon'$ and $\sigma$ and $\Theta_{\psi, W_n^{\epsilon'},V_{n+1}^{\epsilon}}^{\vee}(\pi_{n+1})$ (and thus $\Theta_{\psi, W_n^{\epsilon'},V_{n+1}^{\epsilon}}(\pi_{n+1})$) are uniquely determined by $\phi_2^{\vee}$ and $\phi_1$. By Theorem \ref{Tae} (iv), we know that $\pi_{n+1}=\Theta_{\psi, V_{n+1}^{\epsilon},W_n^{\epsilon'}}(\Theta_{\psi, W_n^{\epsilon'},V_{n+1}^{\epsilon}}(\pi_{n+1}))$. \\
Thus $\epsilon=\epsilon' \cdot \epsilon(\half,\phi_1\otimes \chi_{V}^{-1},\psi_2^E)$, $\pi_{n+1}$ and $\pi_{n+2}=\Theta_{\psi, V_{n+2}^{\epsilon},W_n^{\epsilon'}}(\sigma)$ must have been already determined by $\phi_2^{\vee}$ and $\phi_1$.\\

\item  If $n$ is even, put $\epsilon=\epsilon' \cdot \epsilon(\half,\phi_1^{\vee}\otimes \chi_{V},\psi_2^E)$.  Then by (i) and (iii) of Theorem \ref{theta}, we can write $\pi_{n+2}=\Theta_{\psi,\chi,V_{n+2},W_n}(\sigma)$ for some $\sigma \in \Pi^{\epsilon'}_{\phi_1}$.  Using the see-saw identity, one has $$\Hom_{U(V_{n+1}^{\epsilon})}(\pi_{n+2},\pi_{n+1})\cong \Hom_{U(W_n^{\epsilon'})}(\Theta_{\psi, ,W_n^{\epsilon'},V_{n+1}^{\epsilon}}(\pi_{n+1}) \otimes \omega^{\vee}_{\psi,\chi,W_n^{\epsilon'}},\sigma )\ne 0.$$ 
Since  $\sigma$ is unitary, we have $$\Hom_{U(W_n^{\epsilon'})}(\Theta_{\psi, W_n^{\epsilon'},V_{n+1}^{\epsilon}}(\pi_{n+1})\otimes \sigma^{\vee},
\omega_{\psi,\chi,W_n^{\epsilon'}})\ne0$$ and by Theorem \ref{Tae} (iv), the $L$-parameter of $\Theta_{\psi, W_n^{\epsilon'},V_{n+1}^{\epsilon}}(\pi_{n+1})$ is $\phi_2$. Thus by (FJ)$_n$, we see that $\epsilon'$ and $\sigma^{\vee}$ and $\Theta_{\psi, W_n^{\epsilon'},V_{n+1}^{\epsilon}}(\pi_{n+1})$ are uniquely determined by $\phi_2$ and $\phi_1^{\vee}$. By Theorem \ref{Tae} (iv), we know that $\pi_{n+1}=\Theta_{\psi, V_{n+1}^{\epsilon},W_n^{\epsilon'}}(\Theta_{\psi, W_n^{\epsilon'},V_{n+1}^{\epsilon}}(\pi_{n+1}))$. \\
Thus $\epsilon=\epsilon' \cdot \epsilon(\half,\phi_1^{\vee}\otimes \chi_{V},\psi_2^E)$, $\pi_{n+1}$ and $\pi_{n+2}=\Theta_{\psi, V_{n+2}^{\epsilon},W_n^{\epsilon'}}(\sigma)$ must have been already determined by $\phi_2$ and $\phi_1^{\vee}$.
\end{itemize}
Thus we see that the pair of representations in $\Pi_{\theta(\phi_1)} \times \Pi_{\theta(\phi_{2})}$
$$(\pi_{n+2},\pi_{n+1})=\begin{cases}(\Theta_{\psi,V^{\epsilon}_{n+2},W^{\epsilon'}_{n}}(\pi_{\phi_1}),\Theta_{\psi,V^{\epsilon}_{n+1},W^{\epsilon'}_n}\big((\pi_{\phi^{\vee}_2})^{\vee}\big)),    \quad \quad \quad \quad \text {\ if $n$ is odd}  \\( \big(\Theta_{\psi,V^{\epsilon}_{n+2},W^{\epsilon'}_{n}}(\pi^{\vee}_{\phi_1})\big)^{\vee} ,\big(\Theta_{\psi,V^{\epsilon}_{n+1},W^{\epsilon'}_n}(\pi_{\phi^{\vee}_2})\big)^{\vee}) , \quad \quad \quad \text {if $n$ is even} \end{cases}$$
we found in the existence part is the unique one which makes $ \dim_{\CC}\Hom_{U(V_{n+1}^{\pm})}(\pi_{n+2},\pi_{n+1}) =1$. \\From the recipe of (FJ)$_{n}$ and Theorem \ref{Tae} (ii) and Theorem \ref{theta} (iii), we can easily check that  their associated characters are as described in (\ref{etao}) and (\ref{etae}) .\\
\noindent (Note that $\eta^{\diamondsuit}(z_{\theta(\phi_1)})=\eta^{\heartsuit}(z_{\theta(\phi_2)}).$)

\end{proof}


\begin{rem}Even when $\phi_2$ contatins $\chi_{V}\chi^{(-1)^{n+1}}$,  we may have (\ref{uni}) under some assumption on $\phi_2$. We record it here as a theorem with its recipe.
\end{rem}

\begin{thm}\label{thm3} Let $\phi_1,\phi_2$ be two tempered $L$-parameters of $U(W_n^{\pm})$ such that $\phi_1$ is a \textbf{SCLP} and $\phi_2$ contain $\chi_{V}\chi^{(-1)^{n+1}}$. We define
$\theta(\phi_1), \theta(\phi_2)$ as in Theorem \ref{thm2}.
We assume that for any $\pi \in \Pi_{\theta(\phi_2)}$, if $\Theta_{\psi, W_n^{\epsilon'},V_{n+1}^{\epsilon}}(\pi)$ is non-zero, it is irreducible.
 \noindent Then \beqna  \label{uni2}\sum_{(\pi_{n+2}, \pi_{n+1})\in \Pi^{\pm}_{\theta(\phi_1)} \times \Pi^{\pm}_{\theta(\phi_2)}}\dim_{\CC}\Hom_{U(V_{n+1}^{\pm})}(\pi_{n+2},\pi_{n+1}) =1\eeqna and we can explicate $(\pi_{n+2}, \pi_{n+1})\in \Pi^{\pm}_{\theta(\phi_1)} \times \Pi^{\pm}_{\theta(\phi_2)}$ such that $\dim_{\CC}\Hom_{U(V_{n+1}^{\pm})}(\pi_{n+2},\pi_{n+1})=1$ as follows:\\ \\
Write $S_{\phi_1}=S_{\theta(\phi_1)}= \prod_i  (\ZZ /2 \ZZ) a_i, \quad \quad \quad  S_{\phi_2}=S_{\theta(\phi_2)}= \prod_j  (\ZZ / 2\ZZ) b_j.$

\noindent Then for $(\pi_{n+2},\pi_{n+1})\in \Pi_{\theta(\phi_1)} \times \Pi_{\theta(\phi_2)},$ 
$$\Hom_{U(V_{n+1}^{\epsilon})}(\pi_{n+2},\pi_{n+1})\ne0 \Leftrightarrow (\pi_{n+2},\pi_{n+1})=(\pi_{\theta(\phi_1)}(\eta^{\diamondsuit}),\pi_{\theta(\phi_2)}(\eta^{\heartsuit}))$$ where $(\eta^{\diamondsuit},\eta^{\heartsuit})\in \Irr(S_{\theta(\phi_1)}) \times  \Irr(S_{\theta(\phi_2)})$ the pair of characters of the component group is specified as follows;\\

\noindent If $n$ is odd,
\begin{equation}\label{etao2}\begin{cases}\eta^{\diamondsuit}(a_i)=\epsilon(\half, \phi_1^{(i)}\cdot  \chi_{V}^{-1}\chi_W\otimes \phi^{\vee} ,\psi_2^E),\\\eta^{\heartsuit}(b_j)=\epsilon(\frac{1}{2},\phi_1 \otimes  (\phi^{(j)}_2)^{\vee} \otimes \chi^{-1} ,\psi_2^E). \end{cases}\end{equation}  

\noindent If $n$ is even,
\begin{equation}\label{etae2}\begin{cases}\eta^{\diamondsuit}(a_i)=\epsilon(\half, \phi_1^{(i)}\cdot  \chi_{V}^{-1}\chi_W\otimes \phi^{\vee} ,\psi^E),\\\eta^{\heartsuit}(b_j)=\epsilon(\frac{1}{2},\phi_1 \otimes  (\phi^{(j)}_2)^{\vee} \otimes \chi^{-1} ,\psi^E).\end{cases}\end{equation}  

\begin{proof}Since we have already proved the existence part (\ref{exi}) in Theorem \ref{thm2}, it is sufficient  only to prove the uniqueness part. The proof of the uniqueness is essentially same as we have done in (\ref{uni}). In this case, however, we cannot deduce from Proposition 5.4 in \cite{ag} that for $\pi_{n+1}\in \Pi_{\theta(\phi_2)}$, $\Theta_{\psi, W_n^{\epsilon'},V_{n+1}^{\epsilon}}(\pi_{n+1})$ is irreducible and tempered. Instead, it follows from by our assumption and from Proposition 5.5 in \cite{ag}. We omit the detail. 

\end{proof}

\end{thm}

\begin{rem}To make our Theorem \ref{thm3} not vacuous, we give an example of $L$-parameter $\phi_2$ which satisfies the assumption. We consider $\phi_2$ which possesses $\chi_{V}\chi^{(-1)^{n+1}}$ with multiplicity one. Then $\theta(\phi_2)$ contains $\chi_W$ with multiplicity two and we can write $\theta(\phi_2)= 2\chi_W  \oplus \phi_0$ for some tempered $L$-parameter $\phi_0$ not containing $\chi_W$. Then for any $\pi \in \Pi_{\theta(\phi_2)}$, there is a surjective map $\text{Ind}_{P(X_1)}^{U(V_{n+1})}(\chi_W\cdot I_{GL(X_1)}\otimes \pi_0) \twoheadrightarrow \pi$ for some $\pi_0 \in \Pi_{\phi_0}$. (here, $X_1$ is a $1$-dimensional isotropic subspace of $V_{n+1}$, $P(X_{1})$ is the maximal parabolic subgroup of $U(V_{n+1})$ stabilizing $X_1$ and $I_{GL(X_1)}$ is the trivial representation of $GL(X_1)$. For more explanation on the notation, see \cite[\S 5]{ag}) Since $\phi_0$ doesn't contain $\chi_W$, $\Theta_{W_{n-2}, V_{n-1}}(\pi_0)$ is zero and so  by the almost same argument as in \cite[Cor.5.3]{ag} using  \cite[Prop.5.2]{ag}, we get a surjective map $\Theta_{W_{n}, V_{n-1}}(\pi_0) \twoheadrightarrow \Theta_{W_{n}, V_{n+1}}(\pi).$ Thus, if  $\Theta_{W_{n}, V_{n+1}}(\pi)$ is non-zero, then $\Theta_{W_{n}, V_{n-1}}(\pi_0)$ is also non-zero and it is  irreducible by Theorem \ref{Tae} (iii). Then by the above surjective map, we see that $\Theta_{W_{n}, V_{n+1}}(\pi)$ is irreducible.

\end{rem}

\begin{rem} From Theorem \ref{thm2} and Theorem \ref{thm3}, we see that the GGP-conjecture is no longer true for non-generic $L$-parameters. But even in the non-generic case, if two $L$-parameters are closely related to each other, both theorems hints the existence of the generalized GGP type formula. So, it would be very interesting to find a extended version of GGP conjecture including both generic and non-generic cases.

\end{rem}

\subsection*{Acknowledgements} The dept that this paper owes to the work of Gan and Ichino should be evident to the reader. The author would like to thank to Hiraku Atobe for many invaluable comments on the first draft of this paper. We also thank to KIAS for providing me a wonderful place to conduct a research and its generous supports.

\end{document}